\documentclass[12pt]{iopart}
\usepackage{iopams}
\usepackage{graphicx}
\usepackage{amssymb}
\usepackage{tabularx}

\def\tfrac#1#2{\textstyle\frac{#1}{#2}}

\begin{document}

\title[On stable estimation of flow geometry and  wall shear stress]{On the stable estimation of flow geometry and wall shear stress from magnetic resonance images}

\author{H. Egger and G. Teschner}
\address{AG Numerik, TU Darmstadt, Dolivostra{\ss}e 15, 62393 Darmstadt, Germany}
\ead{egger,teschner@mathematik.tu-darmstadt.de}

\newtheorem{theorem}{Theorem}
\newtheorem{lemma}[theorem]{Lemma}
\newenvironment{proof}{Proof.}{\hfill$\square$\newline}
\newtheorem{remark}[theorem]{Remark}

\renewcommand{\thetheorem}{\arabic{section}.\arabic{theorem}.}

\begin{abstract}
We consider the stable reconstruction of flow geometry and wall shear stress from measurements obtained by magnetic resonance imaging. As noted in a review article by Petersson, most approaches considered so far in the literature seem not be satisfactory. We therefore propose a systematic reconstruction procedure that allows to obtain stable estimates of flow geometry and wall shear stress and we are able to quantify the reconstruction errors in terms of bounds for the measurement 
errors under reasonable smoothness assumptions. A full analysis of the approach is given in the framework of regularization methods. 
In addition, we discuss the efficient implementation of our method and we demonstrate its viability, accuracy, and regularizing properties for experimental data.
\end{abstract}

\vspace{2pc}
\noindent{\it Keywords}: wall shear stress estimation, magnetic resonance imaging, ill-posed problem, Tikhonov regularization, conditional stability 

\section{Introduction}

Due to its clinical relevance, the estimation of wall shear stress in arteries, i.e., of the normal derivative of tangential velocity components at aterial walls, has attracted significant interest in the literature; see e.g. \cite{Gimbrone97,MarklEtAl12,Oyre98,Petersson12,PottersEtAl15,Stalder08} and the references therein. 
Unfortunately, most of the approaches utilized so far suffer at least from one of the following artifacts:
\begin{itemize}
 \item wall shear stress is systematically underestimated for coarse data resolution;
 \item for increasing data resolution the estimates become increasingly unstable.
\end{itemize}
Let us briefly discuss some reasons for these problems:
Internal measurements are used in \cite{Oyre98} to fit cubic polynomials to the flow profile. In the presence of a logarithmic turbulent boundary layer \cite{Khoury13}, this leads to a flattened 
approximation of the velocity near the boundary, thus underestimating the velocity derivative 
and overestimating the radius of the flow region.
Due to a kink of the velocity profile at the boundary, a spline interpolation in the whole domain,
as proposed in \cite{Stalder08}, leads to inaccurate representation of the velocity, in particular, 
near the boundary, which makes the estimate of wall shear stress unreliable.
These observations are amplified by the fact that noise in the velocity measurements can be expected to be particularly high close to the boundary; 
compare with the data in Section~\ref{sec:num} and the remarks in the appendix.
Let us note that even for exact velocity data the evaluation of the wall shear stress is unstable with respect to the boundary location due to the discontinuity of the normal derivative of velocity at the boundary.
These observations made the authors of \cite{Petersson12} conclude that {\em even in the absence of noise and for relatively simple velocity profiles, all methods evaluated were found to be impacted by considerable errors}.

In this paper, we investigate the stable estimation of flow geometry, velocity, and wall shear stress 
by a problem adapted approach that overcomes the above difficulties and that allows for a rigorous 
stability and convergence analysis.
The overall reconstruction problem will be decomposed into the following three natural steps:
\begin{itemize}
 \item estimation of the flow boundary from magnetic resonance images (geometry identification); 
 \item reconstruction of flow velocity from phase contrast images in a function class that ensures zero velocity at the estimated flow boundary (velocity estimation);
 \item evaluation of the normal derivative of the velocity at the boundary
       (wall shear stress computation).      
\end{itemize}
In principle, various methods for the solution of the three sub-problems are available. We here consider one particular approach that allows to conduct a full convergence analysis of the overall reconstruction process under reasonable smoothness assumptions.

For the geometry identification problem, we utilize a parametric formulation which leads to a nonlinear inverse problem with a non-differentiable forward operator. We prove a conditional stability estimate and derive convergence rates for Tikhonov regularization under simple smoothness assumptions on the true geometry. 
The parametrization of the flow domain obtained in the first step is then used to formulate the velocity reconstruction problem on a reference domain, resulting in a linear inverse problem with additional data perturbation stemming from the inexact geometry representation. Again, a full convergence analysis of Tikhonov regularization is presented for this problem. 
By choosing sufficiently strong regularization terms in the first two steps, we obtain stability of the geometry and velocity reconstruction in strong norms that allow us to compute the third step in a stable way. 
Apart from a complete theoretical analysis of our approach, we also discuss the efficient numerical realization 
and demonstrate its viability by application to experimental data.

\section{Geometry identification}

For ease of presentation, we assume in the sequel that the flow geometry is cylindrical and that the flow field has the particular form $(0,0,u(x,y))$, which allows us to restrict the considerations to a two dimensional setting; the extension to three dimensions will be discussed briefly in Section~\ref{sec:3d}.
Without loss of generality, we further assume that measurements are available in the domain $D=(-1,1)^2$, which we call the \emph{field of view}.
 
For a given radius function $R : [0,2\pi] \to \mathbb{R}_+$, we define the domain
\begin{equation} \label{eq:OmegaR}
\Omega_R = \{(r \cos \varphi, r \sin \varphi) : 0 \le r < R(\varphi), \ 0 \le \varphi \le 2\pi \},
\end{equation}
parametrized by $R$, and we denote by $\Omega^\dag = \Omega_{R^\dag}$ the true flow geometry. Throughout the presentation, we will assume that 
\begin{equation} \label{eq:assumption}
  R^\dag \in H^k_{per}(0,2\pi) 
  \mbox{ with } \mbox r_0 \le R(\varphi) \le r_1
\end{equation}
for some constants $2 \le k \le 4$ and $0<r_0 < r_1 < 1$; 
here $H^s_{per}(0,2\pi)$ is the subspace of periodic functions in the Sobolev space $H^s(0,2\pi)$, $s \ge 0$.
Let us note that the above assumptions imply in particular that $\Omega^\dag$ is of class $C^{k-1,\alpha}$, 
compactly embedded in $D$, and star shaped with respect to the origin.

The geometry identification from possibly perturbed magnetic resonance images can now be formulated as a nonlinear inverse problem 
\begin{equation} \label{eq:ip1}
F(R) = m^\delta \qquad \mbox{on } D,
\end{equation}
with forward operator $F:D(F) \subset H^2_{per}(0,2\pi) \to L^2(D)$, $R \mapsto  \chi_{\Omega_R}$ 
defined on the domain $D(F) = \{ R \in H^2_{per}(0,2\pi) : r_0 \le R(\varphi) \le r_1 \}$. 
We further assume that the perturbations in the data are bounded by
\begin{equation} \label{eq:noise}
\|\chi_{\Omega^\dag} - m^\delta\|_{L^2(D)} \le \delta 
\end{equation}
and that knowledge of the noise level $\delta$ is available. 
For the stable approximation of solutions to \eref{eq:ip1}, 
we consider Tikhonov regularization with the functional
\begin{equation} \label{eq:tik1}
 J_\alpha^\delta(R) = \|F(R) - m^\delta\|^2_{L^2(D)} + \alpha \|R\|_{H^2(0,2\pi)}^2.
\end{equation}
We call $R_\alpha^\delta \in D(F)$ an {\em approximate minimizer} for regularization parameter $\alpha>0$, 
if 
\begin{equation} \label{eq:approxmin}
J_\alpha^\delta(R_\alpha^\delta) \le \inf_{R \in D(F)} J(R) + \delta^2.
\end{equation}
Note that the operator $F$ here is continuous but not differentiable and, therefore, standard  convergence rate results about Tikhonov regularization, cf. \cite[Chapter~10]{EnglHankeNeubauer96}, do not apply. 
Instead, we utilize recent results on Tikhonov regularization in Hilbert scales under conditional stability \cite{EggerHofmann18,Tautenhahn98}, which allow us to prove the following assertions.
\begin{theorem} \label{thm:1}
Assume that $R^\dag \in H^k_{per}(0,2\pi) \cap D(F)$, $2 \le k \le 4$.
Then any approximate minimizer $R_\alpha^\delta$ of the Tikhonov functional \eref{eq:tik1} 
with $\alpha=\delta^{4/k}$ satisfies
\begin{eqnarray} \label{eq:rate1}
\fl \qquad \qquad
\|F(R_\alpha^\delta) - \chi_{\Omega^\dag}\|_{L^2(D)} \le C \delta 
\quad \mbox{and} \quad 
\|R_\alpha^\delta - R^\dag\|_{H^r(0,2\pi)} \le C \delta^{1-r/k}
\end{eqnarray}
for $0 \le r \le 2$, with a constant $C$ that only depends on the norm $\|R^\dag\|_{H^k(0,2\pi)}$ of the true solution. 
The same estimates hold true, if $\alpha$ is chosen by a discrepancy principle, i.e., 
\begin{equation} \label{eq:disc}
\alpha = \max \{ \alpha_0\, 2^{-n} : n \ge 0 \mbox{ such that } \|F(R_\alpha^\delta)-m^\delta\|_{L^2(D)}\leq 4\delta \}.
\end{equation}
\end{theorem}
\begin{proof}
Note that $X_s = H^s_{per}(0,2\pi)$, $s \in \mathbb{R}$ defines a Hilbert scale and $Y=L^2(D)$ is a Hilbert space. In view of the results in \cite{EggerHofmann18}, it thus remains to establish a conditional stability estimate for the operator $F$. To do so, let us assume that $R,\widetilde R \in D(F)$
and define $r_{min}(\varphi) = \min \{R(\varphi),\widetilde R(\varphi)\}$ and $r_{max}(\varphi) = \max\{R(\varphi),\widetilde R(\varphi)\}$.
Then 
\begin{equation*}
|F(R) - F(\widetilde R)|(r \cos \varphi,r \sin \varphi) =
\left\{\begin{array}{ll}
1, & \mbox{for } r_{min}(\varphi) \le r \le r_{max}(\varphi), \\ 0 & \mbox{else}.   
\end{array}\right.
\end{equation*}
This allows us to express the difference in the observations by
\begin{eqnarray*}
\fl
\quad \|F(R)-F(\widetilde R)\|_{L^2(D)}^2 
 = \int\nolimits_0^{2\pi}\int\nolimits_0^{r_1} |F(R) - F(\widetilde R)|^2 r\,dr\,d\varphi 
 =  \int\nolimits_0^{2\pi}\int\nolimits_{r_{min}(\varphi)}^{r_{max}(\varphi)}r\,dr\,d\varphi \\
\fl \qquad \qquad 
 = \frac{1}{2}\int\nolimits_0^{2\pi} r_{max}(\varphi)^2 - r_{min}(\varphi)^2 d\varphi
 =  \frac{1}{2} \int\nolimits_0^{2\pi} | R(\varphi)-\widetilde R(\varphi) | \, |R(\varphi) + \widetilde R(\varphi)|\,d\varphi \\
\fl \qquad \qquad 
 \ge \frac{1}{2}\int\nolimits_0^{2\pi} \vert R(\varphi)-\widetilde R(\varphi)| 2 r_0 \frac{|R(\varphi) - \widetilde R(\varphi)|}{|r_1-r_0|} \,d\varphi
  =  \frac{r_0}{r_1-r_0} \|R - \widetilde R\|_{L^2(0,2\pi)}^2.
\end{eqnarray*}
Hence $\|R - \widetilde R\|_{L^2(0,2\pi)} \le C \|F(R) - F(\widetilde R)\|_{L^2(D)}$ for all $R,\widetilde R \in D(F)$, i.e., 
the operator $F$ satisfies a conditional Lipschitz stability estimate.
The assertions of the theorem then follow from Theorems~2.1 and 2.2 in \cite{EggerHofmann18} with $a=0$, $s=2$, $u=k$, and $\gamma=1$.
\end{proof}
\begin{remark} \rm \label{rem:errorbound}
Note that only a simple smoothness condition for the function $R^\dag$, and thus on the domain $\Omega^\dag$, was required to derive a quantitative convergence result here. 
If in addition, an upper bound $\|R^\dag\|_{H^4(0,2\pi)} \le C$ is available, then one obtains
\begin{equation}
\|R_\alpha^\delta - R^\dag\|_{H^2(0,2\pi)} \le \delta_R, 
\end{equation}
with $\delta_R = C \delta^{1/2}$ and we may assume that the bound 
$\delta_R$ is known for further computations.
The results of \cite{EggerHofmann18} even provide more general estimates of the form
\begin{equation*}
\|R_\alpha^\delta - R^\dag\|_{H^r} \le C \delta^{1-r/k}, \qquad 0 \le r \le s,
\end{equation*}
if regularization is performed in the norm of $H^{s}(0,2\pi)$ for $2 \le s \le k/2$  and with regularization parameter $\alpha=\delta^{2s/k}$.
Hence, convergence rates arbitrarily close to one can, in principle, be obtained in any desired norm if the true solution $R^\dag$ 
is sufficiently smooth and the regularization norm is chosen sufficiently strong.
\end{remark}

\section{Velocity approximation}

Let us recall from equation \eref{eq:OmegaR} the definition of a domain $\Omega_R$ parametrized by a radius function $R \in D(F)$. Using the scaling transformation
\begin{eqnarray}
\label{eq:trafo}
\phi_R : (r \cos \varphi,r\sin\varphi) &\mapsto (r_0\,r + (R(\varphi) - r_0)\,r^\eta) \cdot (\cos \varphi, \sin \varphi), 
\end{eqnarray}
with $\eta \ge k \ge 2$ and $k$ as in the previous section, we can express $\Omega_R$ equivalently as image $\Omega_R = \phi_R(B)$ of the unit disk $B = \{ x \in \mathbb{R}^2 : |x|<1\}$. 
Some important properties of this transformation are derived in Appendix~A.
In the following, we assume that
\begin{eqnarray} \label{eq:ass2}
\fl \qquad \qquad 
R^\dag,R_\alpha^\delta \in D(F), \quad R^\dag \in H^4_{per}(0,2\pi), \quad \mbox{and} \quad 
\|R_\alpha^\delta - R^\dag\|_{H^2(0,2\pi)} \le \delta_R, 
\end{eqnarray}
with known bound $\delta_R \le C$; see the discussion in Remark~\ref{rem:errorbound} 
For ease of notation, we write $\phi^\dag = \phi_{R^\dag}$, 
$\phi_\alpha^\delta = \phi_{R_\alpha^\delta}$, 
and denote by $\Omega^\dag = \phi^\dag(B)$ and $\Omega_\alpha^\delta = \phi_\alpha^\delta(B)$ the corresponding domains parametrized by the radius functions $R^\dag$ and $R_\alpha^\delta$, respectively. 

The velocity reconstruction from noisy velocity data $u^\epsilon$ can then be formulated 
compactly as a linear inverse problem over the reference domain, i.e.,
\begin{equation} \label{eq:ip2}
T v = u^\epsilon \circ \phi_\alpha^\delta \qquad \mbox{in } B,
\end{equation}
with operator $T : H^1_0(B)\cap H^2(B) \to L^2(B)$ defined by $T(v)=v$, i.e., as simple embedding between Sobolev spaces. 
We assume that a bound on the data perturbations 
\begin{equation} \label{eq:noise2}
\|u^\dag - u^\epsilon\|_{L^2(D)} \le \epsilon
\end{equation}
is available, where $u^\dag$ denotes the true velocity field to be reconstructed,
and we require that $u^\dag$ is piecewise smooth and vanishes outside 
$\Omega^\dag$, i.e., 
\begin{eqnarray} \label{eq:ass3}
u^\dag \in H^1(D), \qquad  u^\dag|_{\Omega^\dag} \in  H^3(\Omega^\dag), \quad 
\mbox{and} \quad u^\dag|_{D \setminus \Omega^\dag} \equiv 0. 
\end{eqnarray}
Let us note that these assumptions are reasonable, if the flow domain $\Omega^\dag$ is smooth.
\begin{remark} \rm \label{rem:noise}
Observe that information about zero velocity at the boundary of the flow domain is encoded explicitly into the definition of the operator $T$.
The particular formulation \eref{eq:ip2} over the reference domain $B$ is used for the following reason:
Inexact knowledge about the flow domain is shifted to the data $u^\epsilon \circ \phi_\alpha^\delta$ and the operator $T$, therefore, is independent of the geometry reconstruction. This simplifies the analysis given in the following. 
Let us note that invertibility of the transformation $\phi_\alpha^\delta$ is guaranteed by Lemma~\ref{lem:a4}, which allows to associate to any approximate solution $v$ of equation \eref{eq:ip2} a velocity field $u = v \circ (\phi_\alpha^\delta)^{-1}$ in the physical domain. 
\end{remark}

For the stable solution of the nonlinear inverse problem \eref{eq:ip2}, we again consider Tikhonov regularization and we define the regularized approximations by
\begin{equation} \label{eq:tik2}
v_{\alpha,\beta}^{\delta,\epsilon} = \arg\min_{v \in D(T)} \|T v - u^\epsilon \circ \phi_\alpha^\delta\|^2_{L^2(B)} + \beta \|\Delta v\|^2_{L^2(B)}, 
\end{equation}
where $D(T)=H^2(B) \cap H_0^1(B)$ by definition. 
For our further analysis, we will require two auxiliary results that we will present next. 
As a first step, we derive an estimate for the data error in the reference domain. 
\begin{lemma} \label{lem:aux1}
Let the assumptions \eref{eq:ass2}, \eref{eq:noise2} and \eref{eq:ass3} hold. 
Then 
\begin{equation}
\label{eq:error_velocity}
\|u^\dag \circ \phi^\dag - u^\epsilon \circ \phi_\alpha^\delta\|_{L^2(B)} 
\le \delta_U 
\end{equation}
with $\delta_U = C (\delta_R^{1/2}\|u^\dag\|_{H^3(\Omega^\dag)} + \epsilon)$ and a constant $C$ that can be computed explicitly. 
\end{lemma}
\begin{proof}
By the integral transformation theorem we have
\begin{eqnarray*}
\fl \qquad
\|u^\dag \circ \phi^\dag - u^\epsilon \circ \phi_\alpha^\delta\|_{L^2(B)}
&\le&  \|\det (J^\delta_\alpha)^{-1}\|_{L^\infty(B)}\;\| u^\dag \circ \phi^\dag \circ (\phi_\alpha^\delta)^{-1} - u^\epsilon\|_{L^2(\Omega_\alpha^\delta)}\\
&\le& C (\| u^\dag \circ\ \phi^\dag \circ (\phi_\alpha^\delta)^{-1} - u^\dag\|_{L^2(\Omega_\alpha^\delta)}\;+\;\|u^\dag - u^\epsilon\|_{L^2(\Omega_\alpha^\delta)}), 
\end{eqnarray*}
where we used Lemma~\ref{lem:a4} to estimate the Jacobian $J_\alpha^\delta$ of the transformation $\phi_\alpha^\delta$ in the second step.
The last term can then be readily estimated by the bound \eref{eq:noise2} on the data noise.
The first term on the right hand side can further be split as
\begin{eqnarray*}
\fl \qquad
\|u^\dag\circ\phi^\dag\circ (\phi_\alpha^\delta)^{-1} - u^\dag\|_{L^2(\Omega_\alpha^\delta)} \\
\fl \qquad
\le\|u^\dag\circ\phi^\dag\circ(\phi_\alpha^\delta)^{-1} - u^\dag\|_{L^2(\Omega_\alpha^\delta\cap\,\Omega^\dag)} + \|u^\dag\circ\phi^\dag\circ(\phi_\alpha^\delta)^{-1}\|_{L^2(\Omega_\alpha^\delta\setminus\Omega^\dag)} = (i)+(ii),
\end{eqnarray*}
where we used $u^\dag\equiv 0$ in $\Omega_\alpha^\delta\setminus\Omega^\dag$ 
which follows from assumption \eref{eq:ass3}. 
From the smoothness of $R_\alpha^\delta$, $R^\dag$, and $u^\dag$, 
and using the bound on the difference of the inverse transformations $(\phi_\alpha^\delta)^{-1}$ and $(\phi^\dag)^{-1}$ provided by Lemma~\ref{lem:a5}, we then obtain
\begin{eqnarray*}
(i) &=& 
\|u^\dag\circ\phi^\dag\circ(\phi_\alpha^\delta)^{-1} - u^\dag\circ\phi^\dag\circ(\phi^\dag)^{-1}\|_{L^2(\Omega _\alpha^\delta\cap\,\Omega^\dag)}\\
&\le& \|u^\dag\circ\phi^\dag\|_{W^{1,\infty}(B)}\;\|(\phi_\alpha^\delta)^{-1} - (\phi^\dag)^{-1}\|_{L^\infty(\Omega_\alpha^\delta\cap\,\Omega^\dagger)}\;|\Omega_\alpha^\delta\cap\Omega^\dag|^{1/2}\\
&\le& C \|\phi^\dag\|_{W^{1,\infty}(B)}\;\|u^\dag\|_{H^3(\Omega^\dag)}\;\delta_R
\le C \Vert u^\dag\Vert_{H^3(\Omega^\dag)}\;\delta_R.
\end{eqnarray*}
Observe that $\|\phi^\dag\|_{W^{1,\infty}(B)} \le C \|R^\dag\|_{W^{1,\infty}(0,2\pi)} \le C'\|R^\dag\|_{H^2(0,2\pi)}$ by definition \eref{eq:trafo} of the transformation and continuous embedding.
Using Lemma~\ref{lem:a1} and assumption \eref{eq:ass2}, we can control the area $|\Omega_\alpha^\delta \setminus \Omega^\dag|$ of the geometry mismatch by $\delta_R$, which allows us to bound the 
remaining term in the above estimate by
\begin{eqnarray*}
\| u^\dag\circ\phi^\dag\circ(\phi_\alpha^\delta)^{-1}\|_{L^2(\Omega_\alpha^\delta\setminus\Omega^\dag)}
\le
\Vert u^\dag\Vert_{L^\infty(D)}\,|\Omega_\alpha^\delta\setminus\Omega^\dag|^{1/2}
\le
C\,\|u^\dag\|_{H^3(\Omega^\dag)}\,\delta_R^{1/2}.
\end{eqnarray*}
Note that the constants $C$ in the above estimates are generic and independent of 
$\epsilon$ and $\delta$. A combination of the bounds then yields the assertion of the lemma. 
\end{proof}

\begin{remark} \rm \label{rem:errorbound2}
Following Remark~\ref{rem:errorbound} and the arguments used in the proof above, the bound $\delta_U$ in \eref{eq:error_velocity} is computable in terms of the data noise levels $\delta$ and $\epsilon$ in \eref{eq:noise} and \eref{eq:noise2}, if 
bounds on $\|R^\dag\|_{H^4(0,2\pi)}$ and $\|u^\dag\|_{H^3(\Omega^\dag)}$ are available. For our further computations, we may thus assume that $\delta_U$ is known.
\end{remark}

As a next step, we interpret standard smoothness assumptions on $u^\dag$ in terms of the operator $T$, which will allow us to utilized simple source conditions below.
\begin{lemma} \label{lem:aux2}
Let assumptions \eref{eq:ass2} and \eref{eq:ass3} hold and define $v^\dag := u^\dag \circ \phi^\dag$. \\
Then $v^\dag \in R((T^* T)^{\mu})$ for all $\mu < 1/8$. However, $v^\dag \not\in R((T^* T)^{1/8})$, in general.
\end{lemma}
\begin{proof}
We equip $D(T) = H^2(B)\cap H^1_0(B)$ with the norm $\|v\| := \|\Delta v\|_{L^2(B)}$. 
Then for arbitrary $v\in H^2(B)\cap H^1_0(B)$ and $f\in L^2(B)$, we have
\begin{equation*}
(T^*f,v)_{H^2(B) \cap H_0^1(B)}
= (\Delta T^*f,\Delta v)_{L^2(B)} 
= (f,Tv)_{L^2(B)}
=(f,v)_{L^2(B)}.
\end{equation*}
Thus $w = T^* f$ is given as the unique solution of the boundary value problem
\begin{equation*}
\Delta^2 v  =  f \quad  \mbox{in }B 
\qquad \mbox{with} \qquad 
v = 0 \quad \mbox{and} \quad \Delta v  = 0  \qquad  \mbox{on }\partial B.
\end{equation*}
From standard elliptic regularity theory \cite{GilbargTrudinger}, we can conclude that $w = T^*f\in H^4(B)$ for any $f\in L^2(B)$. Using $\mathcal{R}((T^*T)^{1/2}) = \mathcal{R}(T^*)$, see \cite[Thm.~2.6]{EnglHankeNeubauer96}, we thus arrive at
\begin{equation*}
\begin{array}{rcl}
\mathcal{R}((T^*T)^0) & = & \{v\in H^2(B)\mid v\,=\,0\mbox{ on }\partial B\},\\
\mathcal{R}((T^*T)^{1/2}) & = & \{v\in H^4(B)\mid v\,=\Delta v\,=\,0\mbox{ on }\partial B\}.
\end{array}
\end{equation*}
From the regularity assumptions on $u^\dag$ and $R^\dag$, we deduce that $v^\dag \in H^3(B)$ with $v^\dag = 0$ on $\partial B$, but $\Delta v^\dag \neq 0$ on $\partial B$, in general.
By interpolation of Sobolev spaces \cite{Lunardi09}, 
we thus deduce that $v^\dag \in R((T^*T)^\mu)$ for $\mu < 1/8$, but not for $\mu \ge 1/8$, in general. 
\end{proof}
\begin{remark} \rm \label{rem:sobolev} 
A range condition $v^\dag \in \mathcal{R}((T^*T)^\mu)$ would hold with $\mu = 1/4$, if $\Delta v^\dag = 0$ at $\partial B$ would be valid additionally; this can however not be expected in general. 
The limiting factor for the regularity index $\mu$ in the range condition, therefore, is the mismatch of the higher order boundary conditions. 
This could be circumvented by choosing a different equivalent norm on $H^2(B) \cap H_0^1(B)$; 
see \cite{Neubauer88} for details.
\end{remark}

From standard results about Tikhonov regularization in Hilbert spaces for linear inverse problems \cite{EnglHankeNeubauer96}, 
we can now immediately conclude the following results. 
\begin{theorem} \label{thm:velocity}
Let assumptions \eref{eq:ass2}--\eref{eq:ass3} hold 
and let $v_{\alpha,\beta}^{\delta,\epsilon}$ be defined by \eref{eq:tik2} 
with regularization parameter $\beta = \delta_U^{2/(2\mu+1)}$.
Then
\begin{equation}
\|v_{\alpha,\beta}^{\delta,\epsilon} - v^\dag\|_{H^2(B)} \le C \delta_U^{\frac{2\mu}{2\mu+1}}
\quad \mbox{for all} \quad 0 \le \mu < 1/8.
\end{equation}
The same estimates hold for a-posteriori parameter choice by the discrepancy principle
\begin{equation} \label{eq:disc2}
\beta = \max \{ \beta_0\, 2^{-n} : n \ge 0 \mbox{ and such that } \|v_{\alpha,\beta}^{\delta,\epsilon} - v^\dag\| \le 2\delta_U \}.
\end{equation}
\end{theorem}
Corresponding bounds for the error in the velocity $u$ can be obtained, in principle, by back transformation into physical domain and some elementary computations. 
Let us close this section with a remark indicating some natural generalizations.

\begin{remark} \label{rem:errorbound3} \label{rem:physical} \rm
If $u^\dag$ is sufficiently smooth and the functional in \eref{eq:tik2} is replaced by 
\begin{equation*}
\|T v - u^\epsilon \circ \phi_\alpha^\delta\|^2_{L^2(B)} + \beta \|\Delta^t v\|^2_{L^2(B)},
\end{equation*}
with $T: D(T) \subset H^{2t}(B) \cap H_0^1(B) \to L^2(B)$ defined by $T v = v$ and $t \ge 1$ sufficiently large, 
one could, in principle, obtain any rate $2\mu/(2\mu+1)$ sufficiently close to one.
Further note that the data residual could also be measured in physical space. The regularized approximate solution is then defined by 
\begin{equation*}
\widetilde v_{\alpha,\beta}^{\delta,\epsilon} = \arg\min_{v \in D(\widetilde T)} \|\widetilde T v - u^\epsilon\|^2_{L^2(\Omega_\alpha^\delta)} + \beta \|\Delta^t v\|^2_{L^2(B)},
\end{equation*}
with operator $\widetilde T : H^2(B) \cap H_0^1(B) \to L^2(\Omega_\alpha^\delta)$ given by $\widetilde T v = v \circ (\phi_\alpha^\delta)^{-1}$. 
This simply amounts to a change to an equivalent norm in the data space. 
A quick inspection of the arguments in the above proof reveals that the assertions of Theorem~\ref{thm:velocity} remain valid also for this choice of regularization method, which is more easy to implement and will thus be used in our numerical tests in Section~\ref{sec:num}.
\end{remark}

\section{Computation of the wall shear stress}

Let $R\in\mathcal{D}(F)$ be a given radius function with associated transformation $\phi_R$ and 
let $v$ be a given velocity field defined on the reference domain $B$. 
For ease of notation, we assume that fluid viscosity is normalized, and define the associated wall shear stress by 
\begin{equation} \label{eq:wss}
\tau_R(v)(\varphi) =  - n_R(\varphi) \cdot J_R(\cos \varphi,\sin \varphi)^{-1} \cdot (\nabla\,v)(\cos \varphi, \sin \varphi),
\end{equation}
where $J_R$ is the Jacobian of $\phi_R$, and $n_R(\varphi)$ is the outward pointing unit normal vector at the corresponding point $(R(\varphi) \cos \varphi, R(\varphi) \sin \varphi) \in \partial\Omega_R$ in the physical domain. 
Let us note that $n_R(\varphi)$ can be expressed explicitly by 
\begin{equation} \label{eq:normal}
n_R(\varphi) =  \frac{(R(\varphi) \cos\varphi + R'(\varphi) \sin\varphi, R(\varphi) \sin\varphi - R'(\varphi) \cos\varphi)}{\sqrt{(R(\varphi))^2+(R'(\varphi))^2}}.
\end{equation}
For ease of notation, we will identify $\partial B$ with the interval $(0,2\pi)$ in the sequel.
In addition, we again define symbols $\tau_{\alpha,\beta}^{\delta,\epsilon} = \tau_{R_\alpha^\delta}(v_{\alpha,\beta}^{\delta,\epsilon})$ and $\tau^\dag = \tau_{R^\dag}(v^\dag)$ where $v^\dag = u^\dag \circ \phi^\dag$.
A combination of the results derived so far and some elementary geometric computations 
now leads to the following assertion.
\begin{theorem}
Let the assumptions of Theorem 3.7 hold. Then
\begin{equation}
\|\tau_{\alpha,\beta}^{\delta,\epsilon} - \tau^\dag\|_{L^2(0,2\pi)} \le C (\delta_R + \delta_U^{(2\mu)/(2\mu+1)}) \quad \mbox{for all} \quad 0 \le \mu < 1/8.
\end{equation}
\end{theorem}
\begin{proof}
We use the definition of $\tau_R(v)$ and decompose the error into the three parts
\begin{eqnarray*}
\fl\qquad
\|\tau_{\alpha,\beta}^{\delta,\epsilon} - \tau^\dag\|_{L^2(0,2\pi)}
= \|n_{R_\alpha^\delta} J_{R_\alpha^\delta}^{-1} \nabla v_{\alpha,\beta}^{\delta,\epsilon} - n_{R^\dag}\,J_{R^\dag}^{-1}\,\nabla v^\dag\|_{L^2(\partial B)}\\
\fl\qquad\qquad
\le \|(n_{R_\alpha^\delta} - n_{R^\dag}) J_{R_\alpha^\delta}^{-1} \nabla v_{\alpha,\beta}^{\delta,\epsilon}\|_{L^2(\partial B))}
  + \|n_{R^\dag} (J_{R_\alpha^\delta}^{-1} - J_{R^\dag}^{-1})\nabla v_{\alpha,\beta}^{\delta,\epsilon}\|_{L^2(\partial B))} \\
\qquad\qquad\qquad  + \|n_{R^\dag} J_{R^\dag}^{-1} (\nabla v_{\alpha,\beta}^{\delta,\epsilon} - \nabla v^\dag)\|_{L^2(\partial B))} 
= (i) + (ii) + (iii).
\end{eqnarray*}
Using Lemmas~\ref{lem:a2} and \ref{lem:a3}, and the embedding of Sobolev spaces, we obtain
\begin{eqnarray*}
(i) &\le& \|n_{R_\alpha^\delta} - n_{R^\dag}\|_{L^\infty(\partial B))}\|J_{R_\alpha^\delta}^{-1}\|_{L^\infty(\partial B))}\|\nabla v_{\alpha,\beta}^{\delta,\epsilon}\|_{L^2(\partial B))}\\
&\le& C \| R_\alpha^\delta - R^\dag\|_{H^2(0,2\pi)} \|R_\alpha^\delta\|_{H^2(0,2\pi)}\;\|v_{\alpha,\beta}^{\delta,\epsilon}\|_{H^2(B)}
\le C\delta_R.
\end{eqnarray*}
By the geometric arguments of Lemma~\ref{lem:a5}, the second term can be estimated by
\begin{eqnarray*}
(ii) &\leq& \|n_{R^\dagger}\|_{L^\infty(\partial B))}\,\|J_{R_\alpha^\delta}^{-1} - J_{R^\dag}^{-1}\|_{L^\infty(\partial B))}\,\|\nabla v_{\alpha,\beta}^{\delta,\epsilon}\|_{L^2(\partial B))}\\
&\le& C \|R_\alpha^\delta - R^\dag\|_{H^2(0,2\pi)}\;\|v_{\alpha,\beta}^{\delta,\epsilon}\|_{H^2(B)}
\le C\delta_R.
\end{eqnarray*}
With the help of the results of the previous section, the third term, which measures the amplification of the velocity approximation error, can be bounded by
\begin{eqnarray*}
(iii) &\le& \|n_{R^\dag}\|_{L^\infty(\partial B))} \|J_{R^\dag}^{-1}\|_{L^\infty(\partial B))} \|\nabla v_{\alpha,\beta}^{\delta,\epsilon} - \nabla v^\dag\|_{L^2(\partial B))}\\
&\le& C\;\| R^\dag\|_{H^2(0,2\pi)} \|v_{\alpha,\beta}^{\delta,\epsilon}-v^\dag\|_{H^2(B)}
\le C \delta_U^{(2\mu)/(2\mu+1)}.
\end{eqnarray*}
The assertion of the theorem then follows by combination of these estimates.
\end{proof}
\begin{remark} \rm
Let us note that, following the considerations of Remark~\ref{rem:errorbound} and \ref{rem:errorbound3}, one could in principle again obtain 
convergence rates arbitrarily close to one, if the true flow geometry and velocity are sufficiently smooth and the regularization 
terms in the Tikhonov functionals \eref{eq:tik1} and \eref{eq:tik2} are chosen sufficiently strong. The main observation of the previous theorem therefore is, that it is possible to obtain a quantitative estimate under reasonable smoothness assumptions.
\end{remark}

\section{Remarks on the extension to three dimensions} \label{sec:3d}

For a three-dimensional flow, the wall shear stress is a tensor defined by
\begin{equation*}
\tau=-\mu\left(\nabla u + (\nabla u)^T\right) \cdot n,
\end{equation*}
where $\mu$ is the dynamic viscosity, $u$ is the velocity vector, and $n$ the outer unit normal vector at the vessel wall. 
If appropriate measurements of the density and of all three velocity components are available, then the reconstruction approach and the theoretical results presented in the previous sections can be generalized almost verbatim to the three dimensional setting; only the computational realization becomes more complicated. 
Note that the geometry and velocity reconstruction naturally decompose into a sequence of two-dimensional inverse problems for the individual cross-sections for fixed coordinate $z$ in the flow direction. 
To ensure continuity of the reconstructions with respect to the $z$-coordinate, one has to employ regularization also in the $z$-direction,
which fully couples the problems for the individual cross-sections. 
The additional computational complexity due to this coupling can be overcome by a Kaczmarz strategy \cite{KaltenbacherEtAl08,Rieder14}, 
which allows to reduce the numerical solution to the iterated solution of two-dimensional problems for the individual cross-sections.
A detailed investigation of these computational aspects is, however, not in the scope of the current paper and will be given elsewhere.

\section{Numerical validation} \label{sec:num}

In order to demonstrate the viability of the proposed approach, we now report about the reconstruction of flow geometry, flow velocity,
and wall shear stress from experimental data obtained in a clinical whole-body 3 Tesla magnetic resonance imaging scanner (Prisma, Siemens Healthcare, Erlangen) at the University Medical Center Freiburg. 
As a test case, we consider the flow of water through a cylindrical pipe with constant diameter $d=25.885$mm at a constant flow rate 
of about $6$l/min and a temperature of about $22^\circ$C; details about the experimental setup are described in \cite{BauerEtAl18b}.
Let us note that the flow conditions lead to a Reynolds number of about $5300$ and thus amount to a turbulent flow with steep velocity gradients 
in the boundary layer. 
For our reconstruction procedures, we utilize magnetic resonance images of density and the axial flow velocity acquired with different resolutions corresponding to in plane voxel sizes of $h=1.5$mm to $h=0.3$mm.
Due to the axisymmetic geometry, we expect an almost radially symmetric flow, but this a-priori knowledge will not be utilized in our computational tests. 

Up to a simple translation, the exact geometry $\Omega^\dag=\phi^\dag(B)$ of the pipe is known here. A reference solution $u^{ref}$ for the flow velocity can thus be computed by numerical simulation \cite{Khoury13} and will be used for comparison with our results. 
Let us emphasize that this reference solution represents a time averaged velocity field, in which all temporal fluctuations are suppressed. The experimental data, on the other hand, contain such turbulent fluctuations; see Figure~\ref{fig:vel_raw}. 
In addition, also the flow conditions, e.g., temperature and flow rate, do not match exactly with the simulation data. One therefore cannot expect to get a perfect fit to the simulated reference velocity data.
%

\subsection{Geometry identification} 

For the geometry identification, we utilize the standard magnetic resonance images of the density. After a simple scaling procedure, see Appendix~B for details, the data can be interpreted as
\begin{equation*}
\fl\qquad\qquad
m^\delta|_{V_i} \; = \; \frac{1}{|V_i|}\,\int\limits_{V_i}\,\chi_{\Omega^\dag}(x)\;dx\;+\;\mbox{noise},
\end{equation*}
where $V_i$ denotes the $i$th voxel of the measurement array of size $h\!\times\!h$.
The action of the forward operator for our computational tests is then defined as
\begin{equation}\label{eq:discForw}
\fl\qquad\qquad
F_{\gamma,h}(R)|_{V_i} = \frac{1}{|V_i|}\;\int\limits_{V_i}\,H_\gamma(R(\varphi(\xi))\,-\,|\xi|)\,d\xi.
\end{equation}
where $H_\gamma(x)= \frac{1}{\pi}\arctan\left(\frac{x}{\gamma}\right) + \frac12$, $\gamma>0$, is a smooth approximation of the Heaviside function.
In our computational tests, we thus minimize the Tikhonov functional
\begin{equation*}
\fl\qquad\qquad
J_{\alpha,\gamma}^\delta(R) = \|F_{\gamma,h}(R) - m^\delta\|_{L^2(D)}^2\;+\;\alpha\,\|R\|_{H^2(0,2\pi)}^2
\end{equation*}
over the finite dimensional space of radius functions 
\begin{equation*}
\fl\qquad
V_N =\{R_N\in H^2_{per}(0,2\pi)\,|\,R_N(\varphi) = b_0 + \sum\limits_{k=1}^Na_k\sin(k\varphi)+b_k\cos(k\varphi)\} \cap D(F).
\end{equation*}
Due to the smoothing with $\gamma>0$, the forward operator $F_{h,\gamma}$ is continuously differentiable and a projected 
Gauss--Newton method \cite{EggerLeitao09,KaltenbacherEtAl08} can be used for the minimization process.

For our computational tests, the center of the coordinate system was chosen as the barycenter of the measurement data $m^\delta$
and the true radius was $R^\dag=12.9425$mm.  
In Table~\ref{tab:geo}, we display the relative errors $\|R_{\alpha}^\delta - R^\dag\|_{H^2(0,2\pi)} / \|R^\dag\|_{H^2(0,2\pi)}$ obtained for measurement acquired with different data resolution $h$, and for different choices of the regularization parameter $\alpha$. 

\begin{table}[ht!]
\centering
\setlength{\tabcolsep}{1em}
\begin{tabular}{c||c|c|c|c|c}
     $h\setminus\alpha$  & 0.16 & 0.08 & 0.04 & 0.02 & 0.01 \\
    \hline
    \hline
    1.00 & 0.0220 & 0.0117 & \textbf{0.0098} & 0.0128 & 0.0194\\
    0.75 & 0.0161 & \textbf{0.0111} & 0.0118 & 0.0155 & 0.0215\\
    0.60 & 0.0132 & \textbf{0.0086} & 0.0089 & 0.0129 & 0.0201\\
    0.43 & 0.0104 & \textbf{0.0079} & 0.0090 & 0.0121 & 0.0172\\
    0.30 & 0.0064 & \textbf{0.0053} & 0.0063 & 0.0087 & 0.0127\\
\end{tabular}
\caption{Relative reconstruction errors $\|R_{\alpha}^\delta-R^\dag\|_{H^2(0,2\pi)} / \|R^\dag\|_{H^2(0,2\pi)}$ for different data resolutions $h$ and regularization parameters $\alpha$; optimal results in bold.\label{tab:geo}}
\end{table}

As expected intuitively, reconstruction errors are slightly decreasing with increasing data resolution.  Further note that the reconstructions are stable with respect to the choice of the regularization parameter and the optimal regularization parameters, i.e., those for which the error is minimal, are practically independent of the data resolution. 

In Figure~\ref{fig:geo}, we display one of the data sets used for our computations, together with the reconstructed geometry
and the radius functions $R^\dag$ and $R_{\alpha}^\delta$ obtained in our numerical tests.
\begin{figure}[ht!]
\includegraphics[scale=0.5]{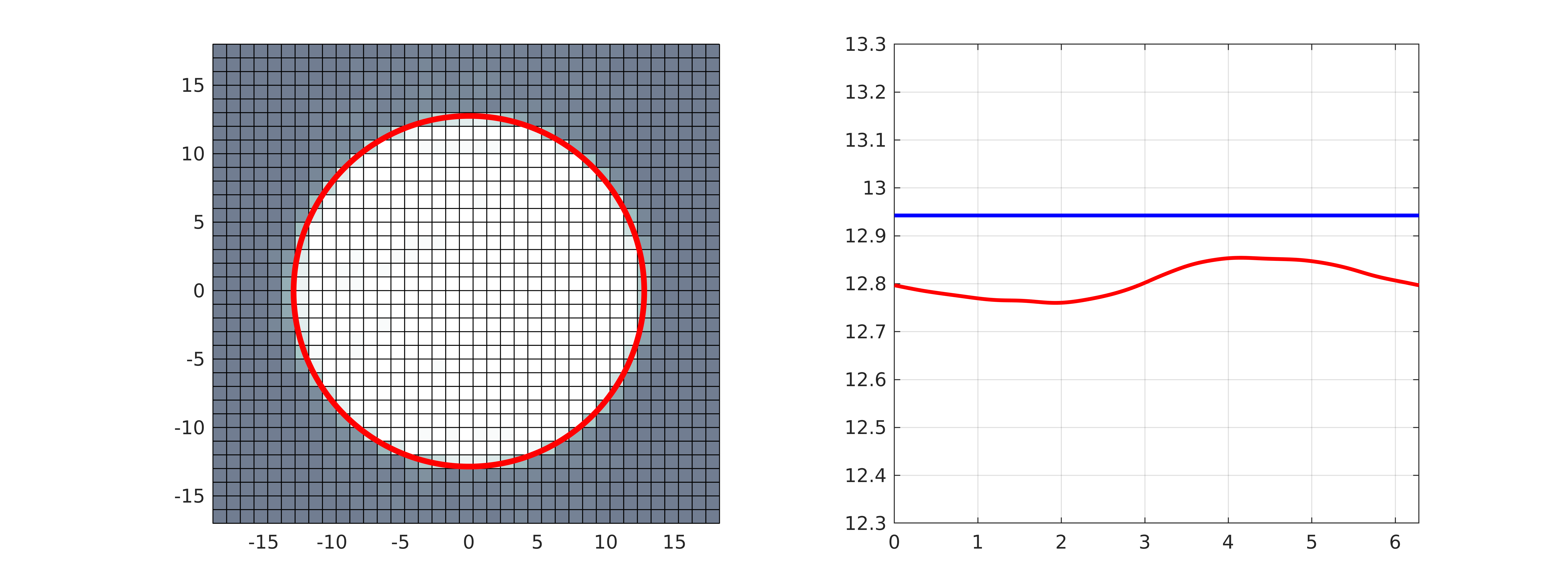}
\caption{Density measurements ($h=1$mm) and reconstructed geometry (left); corresponding radius functions $R^\dag$, $R_{\alpha}^\delta$ (right);
axis labeling in mm. \label{fig:geo}}
\end{figure}
The maximal error in the reconstructed radius is about $0.15$mm, which is substantially less than the voxel size $h=1$mm of the data.
This illustrates that sub-pixel resolution can be obtained by the proposed geometry reconstruction procedure.

\subsection{Velocity approximation}

We now turn to the reconstruction of the velocity field, for which we utilize 
phase-contrast magnetic resonance imaging data \cite{MarklEtAl12,PottersEtAl15}; also see Appendix~B. Let $\Omega_\alpha^\delta = \phi_\alpha^\delta(B)$ denote the approximation of the flow domain obtained with the radius function $R_\alpha^\delta$ reconstructed in the first step with $\alpha$ chosen as the best regularization parameter according to Table~\ref{tab:geo}. For voxels $V_i$ lying at least partially in the flow domain, i.e., with $|V_i \cap \Omega_\alpha^\delta|>0$, the measurements can interpreted as
\begin{equation*}
\fl\qquad\qquad
u^\epsilon|_{V_i}  =  \frac{1}{|V_i\cap\,\Omega_\alpha^\delta|}\,\int\limits_{V_i\cap\,\Omega_\alpha^\delta}\,u^\dag(x)\;dx\;+\;\mbox{noise}.
\end{equation*}
The remaining voxels only contain information about the noise and will not be used in the reconstruction; see Figure~\ref{fig:vel_raw} and the remarks given in Appendix~B.

\begin{figure}[ht!]
\centering
\includegraphics[width=0.8\textwidth]{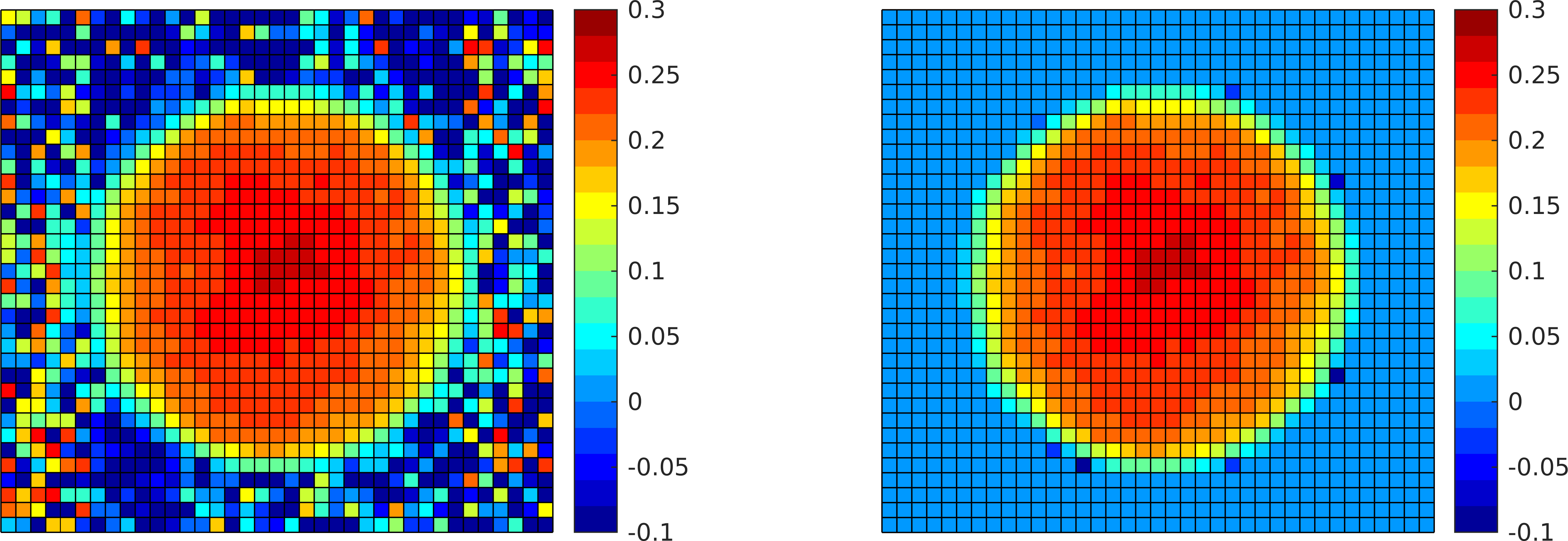}
\caption{Velocity raw data $u^\epsilon$ with large noise outside the flow domain (left) and truncated velocity raw data actually used in the reconstruction (right). \label{fig:vel_raw}}
\end{figure}

Following Remark~\ref{rem:physical}, we define the corresponding forward operator mapping to the physical domain by $\widetilde T_h: H^2(B)\cap H^1_0(B) \rightarrow L^2(\Omega_\alpha^\delta)$ with
\begin{equation*}
\fl\qquad\qquad
\widetilde{T}_h(v)|_{V_i} = \frac{1}{|V_i\cap\,\Omega_\alpha^\delta|}\,\int\limits_{V_i\cap\,\Omega_\alpha^\delta} v\circ(\phi_\alpha^\delta)^{-1}(x)\;dx.
\end{equation*}
Note that only voxels $V_i$ with $|V_i\cap\,\Omega_\alpha^\delta|>0$ are required in the definition of $\widetilde T_h$. 
The regularized approximation for the velocity field in the reference domain is then defined as the minimizer of the Tikhonov functional
\begin{equation*}
\fl\qquad\qquad
\widetilde v_{\alpha,\beta}^{\delta,\epsilon} = \arg\min_{v \in V_M} \|\widetilde T_h v  - u^\epsilon\|_{L^2(\Omega_\alpha^\delta)}^2  +  \beta\|\Delta v\|_{L^2(B)}^2,
\end{equation*}
over the space $V_M = \mbox{span}\{v\in H^1_0(B) \; | \; -\Delta v = \lambda v\mbox{ for some }\lambda\leq M\}$ of eigen functions of the Dirichlet Laplace operator on the unit disk. The solution $\widetilde v_{\alpha,\beta}^{\delta,\epsilon}$ can be computed efficiently by Cholesky factorization or the conjugate gradient method. 
In Table~\ref{tab:vel}, we display the reconstruction errors obtained 
with our algorithm for different data resolutions $h$ and regularization parameters $\beta$.

\begin{table}[ht!]
\centering
\setlength{\tabcolsep}{1em}
\begin{tabular}{c||c|c|c|c|c}
$h\setminus\beta$  & $3.2\cdot10^{-4}$ & $1.6\cdot10^{-4}$ & $8.0\cdot10^{-5}$ & $4.0\cdot10^{-5}$ & $2.0\cdot10^{-5}$ \\
    \hline
    \hline    
    1.00 & 0.3815 & 0.2661 & \textbf{0.1818} & 0.2249 & 0.3985 \\
    0.75 & 0.3317 & 0.2233 & \textbf{0.1926} & 0.2912 & 0.4541 \\
    0.60 & 0.3377 & 0.2313 & \textbf{0.1813} & 0.2507 & 0.3899 \\
    0.43 & 0.3098 & 0.2011 & \textbf{0.1795} & 0.2851 & 0.4478 \\
    0.30 & 0.3245 & 0.2090 & \textbf{0.1518} & 0.2106 & 0.3190
\end{tabular}

\caption{Relative reconstruction errors $\|\widetilde v_{\alpha,\beta}^{\delta,\epsilon} -v^{ref} \|_{H^2(B)}/\|v^{ref} \|_{H^2(B)}$, with $v^{ref}=u^{ref} \circ \phi^\dag$ denoting the simulated velocity on the reference domain, for different data resolutions $h$ and regularization parameters $\beta$; optimal results in bold.\label{tab:vel}} 
\end{table}

Recall that $u^{ref}$, which was obtained by numerical simulation, corresponds to a time averaged velocity field and therefore does not contain temporal fluctuations that are present in the measurements; see Figure~\ref{fig:vel}. 
This explains the relatively large errors in Table~\ref{tab:vel} and their independence of the data resolution. 
The optimal regularization parameters $\beta$ are again independent of the voxel size $h$ used in the data acquisition. 

\begin{figure}[ht!]
\centering
\includegraphics[width=0.8\textwidth]{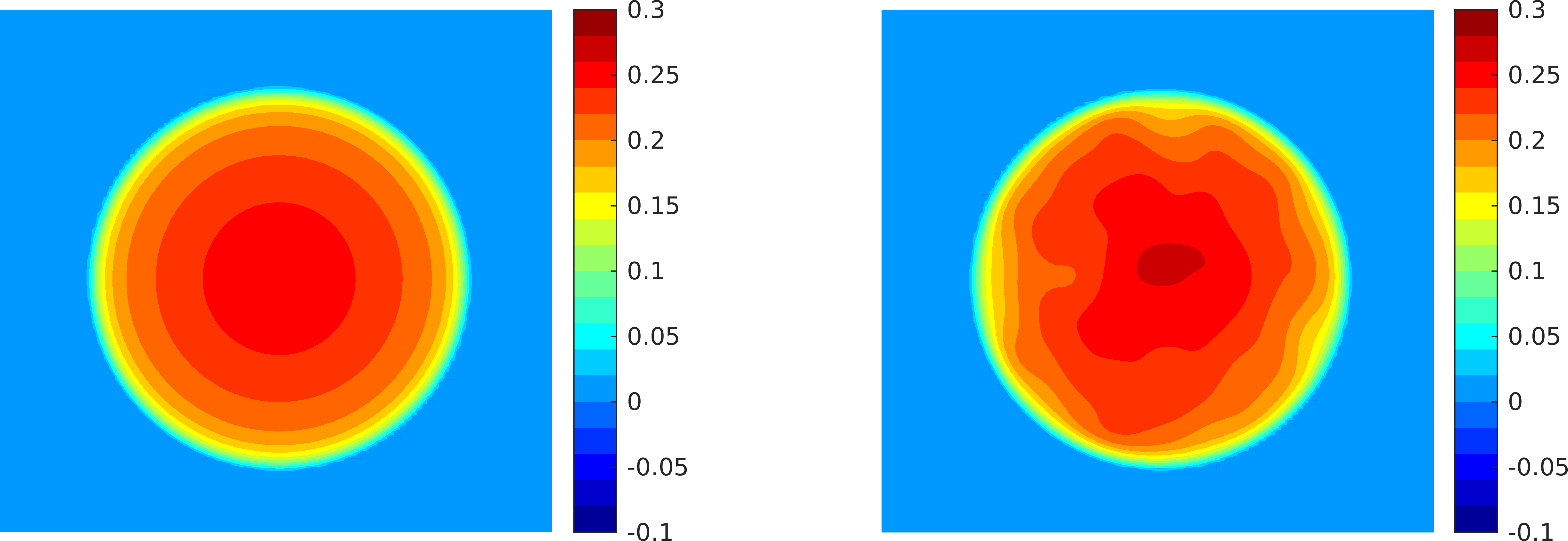}
\caption{Reference velocity $v^{ref} = u^{ref} \circ \phi^\dag$ (left) obtained by simulation and reconstructed velocity $\widetilde v_{\alpha,\beta}^{\delta,\epsilon}$ (right) for data resolution $h=1$mm and $\beta = 8.0\cdot10^{-5}$. \label{fig:vel}}
\end{figure}

A quick comparison with the raw data, depicted in Figure~\ref{fig:vel_raw}, with the plots in Figure~\ref{fig:vel} shows that the reconstructed velocity field correctly reproduces some turbulent fluctuations which, due to time averaging, are not present in the simulated reference velocity field. Apart from these differences, the reconstruction is in good agreement with the reference solution.

\subsection{Computation of the wall shear stress} 

As a final step in our numerical tests, we now utilize the reconstructed radius function $R_{\alpha}^\delta$ and velocity fields $v_{\alpha,\beta}^{\delta,\epsilon}$ to compute the approximation $\tau_{\alpha,\beta}^{\delta,\epsilon}$ for the wall shear stress by \eref{eq:wss}. 
In Table~\ref{tab:wss}, we display the results obtained for optimal $\alpha$ and $\beta = 1.6\cdot10^{-4}$. 
\begin{table}[ht!]
\centering
\begin{tabular}{c||c|c|c|c|c}
h & 1.00 & 0.75 & 0.60 & 0.43 & 0.30 \\ 
\hline
\hline
$\|\tau_{\alpha,\beta}^{\delta,\epsilon} - \tau^\dag\|_{L^2(0,2\pi)}$ & 0.0789 & 0.0385 & 0.0391 & 0.0264 & 0.0405 \\
\end{tabular}
\caption{\label{tab:wss}Relative errors $\|\tau_{\alpha,\beta}^{\delta,\epsilon} - \tau^\dag\|_{L^2(0,2\pi)}$ in the reconstruction of the wall shear stress for different data resolutions resp. voxel size $h$.}
\end{table}
Let us note that most part of the error stems from perturbations in higher modes, which can be suppressed efficiently by application of a low-pass filter. 
In Figure~\ref{fig:wss}, we plot the reconstruction of the wall shear stress $\tau_{\alpha,\beta}^{\delta,\epsilon}$ and its constant approximation $\overline \tau_{\alpha,\beta}^{\delta,\epsilon}$ against the reference value $\tau^{ref}$ obtained from numerical flow simulation \cite{Khoury13}. 

\begin{figure}[ht!]
\centering
\includegraphics[width=0.9\textwidth]{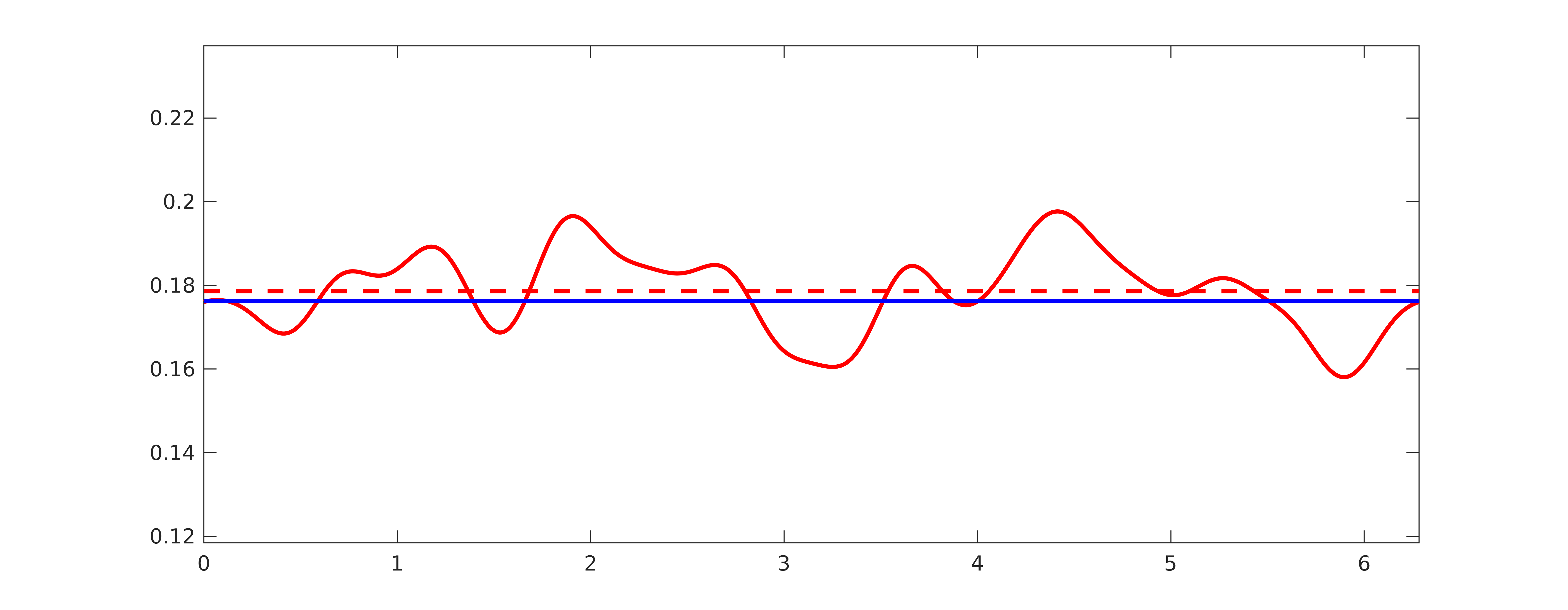}
\caption{Reconstructed wall shear stress $\tau_{\alpha,\beta}^{\delta,\epsilon}$ and constant approximation $\overline \tau_{\alpha,\beta}^{\delta,\epsilon}$ (red) in comparison to the reference wall shear stress $\tau^{ref}$ (blue) obtained by simulation. All results are functions of angle $\varphi$ with values in Pa. \label{fig:wss}}
\end{figure}

Note that the average wall shear stress $\overline \tau_{\alpha,\beta}^{\delta,\epsilon}$ is in very good agreement with the reference value $\tau^{ref}$ obtained for the simulated data. The local variations in the reconstruction $\tau_{\alpha,\beta}^{\delta,\epsilon}$ can be explained by the turbulent variations of velocity in the data; compare with the plots in Figure~\ref{fig:vel_raw} and \ref{fig:vel}.

\section{Discussion}

As reviewed by Petersson \cite{Petersson12}, the stable and accurate estimation of  wall shear stress from magnetic resonance imaging data, is a delicate issue and 
most reconstruction approaches reported in literature seem not to work properly.
In this paper, we therefore considered a systematic approach for the estimation of wall shear stress from magnetic resonance images of density and velocity, for which stability and convergence could be established under simple and realistic smoothness assumptions on the flow geometry and velocity. 
The theoretical results were validated by numerical tests for experimental data which demonstrate that wall shear stress can be estimated from magnetic resonance imaging data with relative errors of a few percent and practically independent of the data resolution; this is in stark contrast to the results of Stalder~\cite{Stalder08}.  

Let us note that stable wall shear stress estimates can also be obtained via empirical Moody charts \cite{Moody44} or the Clauser plot method \cite{Clauser56,ShokinaEtAl18}, which are, however, limited to axisymmetric geometry or fully developed turbulent flow. Numerical simulations could also be used, in principle, to compute wall shear stress estimates \cite{Khoury13}, but precise knowledge about the rheological properties of the fluid are required. 
In contrast to these approaches, the method considered in this paper is generally applicable and, therefore, seems most appropriate for application in a clinical context.

\section*{Acknowledgements}

Experimental data were acquired at the the University Medical Center, Freiburg, 
together with A. Bauer (SLA, TU Darmstadt) and A. Krafft, N. Shokina (Med.Phys., UMC Freiburg). 
Funding of the authors by the German Research Foundation (DFG) via grant Eg-331/1-1 and through grant GSC~233 of the ``Excellence Initiative'' of the German Federal and State Governments is gratefully acknowledged.

\appendix
\renewcommand{\thetheorem}{A.\arabic{theorem}.}
\setcounter{theorem}{0}

\section{Geometric results}

In the following, we present some auxiliary results concerning geometrical details.
Recall that $D(F)=\{R\in H^2_{per}(0,2\pi)\mid r_0<R<r_1\}$ for some $0<r_0<r_1<1$ and that $\Omega_R = \phi_R(B)$ is the image of the unit ball $B$ under the transformation
\begin{equation*}
\phi_R(r \cos(\varphi),r\sin(\varphi)) = (r_0 + (R(\varphi) - r_0)\,r^\eta)\;(\cos(\varphi),\sin(\varphi))
\end{equation*}
with $\eta\geq 2$. The Jacobian matrix of $\phi_R$ is denoted by $J_R$, and for given functions $R_1,R_2$, the subscripts transfer to the associated objects, e.g. $\phi_1=\phi_{R_1}$ and $J_2 = J_{R_2}$.

As a first result, we estimate the differences of domains $\Omega_R = \phi_R(B)$ in terms of differences of their parameterizing radius functions.
\begin{lemma}\label{lem:a1}
Let $R_1,R_2\in D(F)$. Then
\begin{equation} \label{eq:domain_diff}
|\Omega_1\setminus\Omega_2| \leq C \|R_1-R_2\|_{H^1(0,2\pi)}.
\end{equation}
\end{lemma}
\begin{proof}
Let us define $R_{max}(\varphi) = \max\{R_1(\varphi),R_2(\varphi)\}$. Then 
\begin{eqnarray*}
\fl\quad
|\Omega_1\setminus\Omega_2|
=\int\nolimits_0^{2\pi}\int\nolimits_{R_2(\varphi)}^{R_{max}(\varphi)}r\,dr\,d\varphi
=\frac{1}{2}\int\nolimits_0^{2\pi}(R_{max}(\varphi)+R_2(\varphi)) (R_{max}(\varphi)-R_2(\varphi))\,d\varphi\\
\fl\qquad\qquad
\leq 2 \pi r_1 \|R_{max} - R_2\|_{L^\infty(0,2\pi)} 
\leq 2 \pi r_1 C \|R_1-R_2\|_{H^1(0,2\pi)}.
\end{eqnarray*}
In the last step, we used the continuous embedding $H^1(0,2\pi) \hookrightarrow L^\infty(0,2\pi)$. 
\end{proof}

As a next step, we estimate differences in the normal vector $n_R$ defined in \eref{eq:normal}, in terms of differences in the radius function.
\begin{lemma}\label{lem:a2}
Let $R_1,R_2\in D(F)$. Then
\begin{equation}
\label{eq:normal_diff}
\|n_1 - n_2\|_{L^\infty(\partial B)} \leq C\,\|R_1-R_2\|_{H^2(0,2\pi)}.
\end{equation}
\end{lemma}
\begin{proof}
Let us introduce $e_r = (\cos(\varphi),\sin(\varphi))$ and $e_\varphi = (-\sin(\varphi),\cos(\varphi))$. 
Omitting the explicit notion of the dependence on $\varphi$, we then obtain for any angle $\varphi$ that 
\begin{equation*}
|n_1- n_2|^2 = |\lambda_r\,e_r + \lambda_\varphi\,e_\varphi|^2 = \lambda_r^2 + \lambda_\varphi^2.
\end{equation*}
The radial component of the difference can be estimated by 
\begin{eqnarray*}
\fl\quad
\left|\lambda_r\right| 
=  \left|\tfrac{R_1}{\sqrt{R_1^2 + R_1^{\prime 2}}} - \tfrac{R_2}{\sqrt{R_2^2 + R_2^{\prime 2}}}\right|
\leq \left|\tfrac{R_1-R_2}{\sqrt{R_1^2 + R_1^{\prime 2}}} + R_2\tfrac{\sqrt{R_2^2 + R_2^{\prime 2}}-\sqrt{R_1^2 + R_1^{\prime 2}}}{\sqrt{R_1^2 + R_1^{\prime 2}}\sqrt{R_2^2 + R_2^{\prime 2}}}\right|\\
\fl\quad\quad\quad
\leq \tfrac{|R_1-R_2|}{r_0} + R_2\tfrac{\sqrt{(R_1-R_2)^2 + (R_1^\prime-R_2^\prime)^2}}{\sqrt{R_1^2 + R_1^{\prime 2}}\sqrt{R_2^2 + R_2^{\prime 2}}}
\leq \tfrac{1 + \sqrt2}{r_0}\|R_1-R_2\|_{W^{1,\infty}(0,2\pi)}
\end{eqnarray*}
In a similar way, one can estimate $\lambda_\varphi$, and by Sobolev's embedding theorem, we obtain \begin{eqnarray*}
\|n_1-n_2\|_{L^\infty(\partial B)}
\leq \|\lambda_r\|_{L^\infty(\partial B)} + \|\lambda_\varphi\|_{L^\infty(\partial B)}
\leq C \|R_1-R_2\|_{H^2(0,2\pi)},
\end{eqnarray*}
which already yields the desired estimate.
\end{proof}

The following result ensures smoothness of the transformations $\phi_R$ 
whenever the radius function $R$ is sufficiently smooth. 
\begin{lemma}\label{lem:a3}
Let $R_1,R_2\in D(F) \cap H^k_{per}(0,2\pi)$ for $k\le \eta$ with $\eta \ge 2$ as in \eref{eq:trafo}.
Then
\begin{equation}
\label{eq:phi_continuous}
\|\phi_1\|_{W^{k-1,\infty}(B)} \leq C\;\|R_1\|_{H^k(0,2\pi)}.
\end{equation}
where $\phi_1 = \phi_{R_1}$ with $\phi_R$ defined in \eref{eq:trafo}. Moreover, 
\begin{equation}
\label{eq:diff_phi_continuous}
\|\phi_1 - \phi_2\|_{W^{k-1,\infty}(B)} \leq C\;\|R_1-R_2\|_{H^k(0,2\pi)}
\end{equation}
\end{lemma}
\begin{proof}
The continuity of $\phi_1$ is obvious and the Jacobian $J_1=J_{R_1}$ obtained by derivation with respect to coordinates $x=(r \cos \phi,r \sin \phi)$ reads
\begin{equation}
\label{eq:Jacobian_rep}
\fl
J_1(r \cos \varphi, r \sin \varphi) 
= \left(
\begin{array}{cc}
r_0\,+\,\eta\,(R_1(\varphi)-r_0)\,r^{\eta-1} & R_1^\prime(\varphi)\,r^{\eta-1}\\
0 & r_0\,+\,(R_1(\varphi)-r_0)\,r^{\eta-1}
\end{array}
\right).
\end{equation}
Since $\eta \ge 2$, the Jacobian can be seen to be continuous also in $r=0$, and we have
\begin{equation*}
\| J_1\|_{L^\infty(B)} 
\leq C\,(\|R_1\|_{L^\infty(0,2\pi)} + \|R_1^\prime\|_{L^\infty(0,2\pi)})
\leq C\|R_1\|_{H^2(0,2\pi)}.
\end{equation*}
This shows the first estimate for $k=1$. 
The assertion for higher order derivatives of $\phi_1$ follow in a similar way. 
From the formula \eref{eq:trafo}, one can see that $\phi_R$ is affine linear in $R$. Hence the second estimate follows directly from the first. 
\end{proof}

As a next step, we show that the transformations $\phi_R$ defined in \eref{eq:trafo} are invertible.
\begin{lemma}\label{lem:a4}
Let $R\in D(F)$. 
Then the transformation $\phi_R$ defined in \eref{eq:trafo} 
is a diffeomorphism with inverse transformation $\psi_R = (\phi_R)^{-1}$ and
\begin{equation}
\label{eq:BoundJacInverse}
\|J_R^{-1}\|_{L^\infty(B)} = \| J_{\psi_R}\|_{L^\infty(\Omega_R)}\leq C\,\|R\|_{H^2(0,2\pi)}
\end{equation}
\end{lemma}
\begin{proof}
Using \eref{eq:Jacobian_rep}, we can estimate the determinant by
\begin{equation*}
\fl\qquad
\det(J_R(r \cos\varphi,r\sin\varphi)) = (r_0+\eta(R(\varphi)-r_0)r^{\eta-1})\cdot(r_0 + (R(\varphi)-r_0)r^{\eta-1})\ge r_0^2.
\end{equation*}
The existence of an inverse transformation $\psi_R = (\phi_R)^{-1}$ then follows form the 
implicit function theorem and the Jacobian of the inverse mapping is 
$J_{\psi_R}\circ\phi_R = J_R^{-1}$. The bounds for $J_{\psi_R}$ can then be deduced in an elementary way.
\end{proof}

Using the previous results, we can also bound differences in the inverse mappings.
\begin{lemma}\label{lem:a5}
Let $R_1,R_2\in D(F)$ and assume that $R_1\in H^3_{per}(0,2\pi)$. 
Furthermore, let $\psi_1$, $\psi_2$ denote the corresponding inverse transformations with Jacobians $J_{\psi_1}$, $J_{\psi_2}$. Then 
\begin{equation}
\label{eq:diff_inverse}
\|\psi_1-\psi_2\|_{L^\infty(\Omega_1\cap\,\Omega_2)} \leq C\, \|R_1-R_2\|_{H^1(0,2\pi)}
\end{equation}
and the difference in the Jacobians can be bounded by
\begin{equation}
\label{eq:diff_inverse_derivative}
\|J_{\psi_1}-J_{\psi_2}\|_{L^\infty(\Omega_1\cap\,\Omega_2)} \leq C\, \|R_1-R_2\|_{H^2(0,2\pi)}
\end{equation}
\end{lemma}
\begin{proof}
\textit{Step 1:} To show \eref{eq:diff_inverse}, let $x\in\Omega_1\cap\,\Omega_2$. Then there are $x_1,x_2\in B$ with $\phi_1(x_1) = \phi_2(x_2) = x$. Since the transformations $\phi_1$ and $\phi_2$ preserve angles, the angular parts $\varphi(x_1) = \varphi(x_2) = \varphi(x) =:\varphi$
are equal. $\phi_2$ has an inverse $\psi_2$, hence
\begin{equation*}
\fl\quad
x_2-x_1 \; = \; \psi_2(\phi_2(x_2))\,-\,\psi_2(\phi_2(x_1)) \; = \; \psi_2(x)\,-\,\psi_2(x\,+dx),
\end{equation*}
where the defect is given by
\begin{equation*}
\fl\quad
dx \; = \; \phi_2(x_1)\,-\,\phi_1(x_1) \; = \; (R_2(\varphi) \, - \, R_1(\varphi))\;|x_1|^\eta\;(\cos(\varphi),\sin(\varphi)).
\end{equation*}
Since $x$ and $dx$ have the same angular coordinate and $\Omega_2$ is star shaped with respect to the origin, we have $\{x + t\;dx\;|\;t\in[0,1]\}\subset\Omega_2$. The mean value theorem yields
\begin{equation*}
\fl\quad
|\psi_2(x)-\psi_1(x)|\;=\;|x_2-x_1| \; = \; |J_{\psi_2}(\xi)\;dx|\;\leq\;\|J_{\psi_2}\|_{W^{1,\infty}(\Omega_2)}\;\|R_1 - R_2\|_{H^1(0,2\pi)}.
\end{equation*}
The assertion \eref{eq:diff_inverse} follows by the estimate of Lemma~\ref{lem:a4}
\textit{Step 2: Show \eref{eq:diff_inverse_derivative}.} Starting at \eref{eq:Jacobian_rep} elementary calculus yields
\begin{equation*}
\fl\quad
\|\det J_{\phi_1}\,-\,\det J_{\phi_2}\|_{L^\infty(B)}\;\leq\;C\,\|R_1-R_2\|_{L^\infty(0,2\pi)},
\end{equation*}
where the constant $C$ depends only on $\|R_1\|_{L^\infty(0,2\pi)}$, $\|R_2\|_{L^\infty(0,2\pi)}$, $r_0$ and $\eta$. 
For the inverse Jacobians, we further conclude that
\begin{equation*}
\fl\quad
\begin{array}{rcl}
\|J_{\phi_1}^{-1} - J_{\phi_2}^{-1}\|_{L^\infty(B)} & \leq & \left\|\frac{1}{\det J_{\phi_1}}J_{\phi_1} - \frac{1}{\det J_{\phi_2}}J_{\phi_2}\right\|_{L^\infty(B)}\\[2ex]
& \leq & \left\|\frac{\det J_{\phi_2}\; - \;\det J_{\phi_1}}{\det J_{\phi_1}\;\det J_{\phi_2}}\;J_{\phi_1}\right\|_{L^\infty(B)}\;+\;\left\|\frac{1}{\det J_{\phi_2}}\;(J_{\phi_1}-J_{\phi_2})\right\|_{L^\infty(B)}\\[2ex]
& \leq & C\;\left(\frac{1}{r_0^4}\;\|R_1\|_{H^2(0,2\pi)}\ + \frac{1}{r_0^2}\right)\;\|R_1-R_2\|_{H^2(0,2\pi)}.
\end{array}
\end{equation*}
Since the inverse Jacobian of $\phi_1$ has the representation $J_{\phi_1}^{-1} = \tilde{J}_{\phi_1}/\det J_{\phi_1}\;$, where $\tilde{J}_{\phi_1}$ is a rearrangement of $J_{\phi_1}$ and all expressions are continuously differentiable, one can verify without difficulty that $J^{-1}_{\phi_1}\in W^{1,\infty}(B)$ with the associated norm bounded in terms of $\eta$, $r_0$, and $\|R_1\|_{H^3(0,2\pi)}$. Therefore we arrive at \eref{eq:diff_inverse_derivative} by
\begin{equation*}
\fl\quad
\begin{array}{rcl}
\|J_{\psi_1}-J_{\psi_2}\|_{L^\infty(\Omega_1\cap\,\Omega_2)} & = & \|J_{\phi_1}^{-1}\circ\psi_1\,-\,J_{\phi_2}^{-1}\circ\psi_2\|_{L^\infty(\Omega_1\cap\,\Omega_2)}\\[1ex]
& \leq & \|J_{\phi_1}^{-1}\|_{W^{1,\infty}(B)}\,\|\psi_1-\psi_2\|_{L^\infty(\Omega_1\cap\,\Omega_2)}\;+\;\|J_{\phi_1}^{-1}-J_{\phi_2}^{-1}\|_{L^\infty(B)}\\[1ex]
& \leq & C\|R_1 - R_2\|_{H^2(0,2\pi)}.
\end{array}
\end{equation*}
This completes the proof of the second estimate of the lemma.
\end{proof}

\section{Interpretation of the data}

Let us briefly comment on the physical interpretation and the preprocessing of the experimental data used for the reconstructions in Section~\ref{sec:num}.

\subsection{Magnitude data}

The magnitude raw data $m^\delta_{raw}$ represent integral means of the proton density over voxels; this values are additionally perturbed by data noise. 
From a histogram of the magnitude data, see Figure~B1, one can deduce to peak values $m_0$ and $m_1$, which are used for scaling of the data.

\begin{figure}[ht!]
\fl\qquad
\label{fig:MAG_HIST}
\includegraphics[scale=0.6]{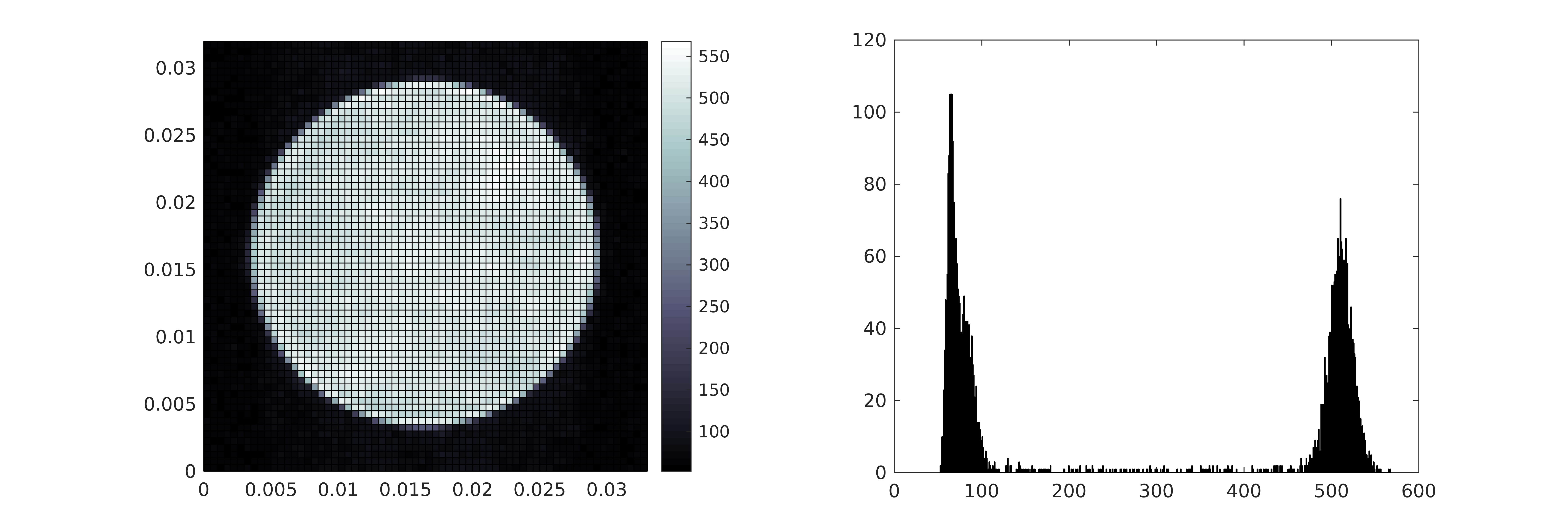}
\caption{Magnitude raw data (left) and magnitude value histogram (right) used to normalize the magnitude data.}
\end{figure}

\noindent
Based on the two values $m_0$, $m_1$, 
the normalized magnitude $m^\delta$ are then defined by
\begin{equation*}
\fl\qquad\qquad
m^\delta|_{V_i} \; = \; T\left(\frac{m_{raw}^\delta|_{V_i} - m_0}{m_1 - m_0}\right),\end{equation*}
with truncation function $T(x) = \max(0,\min(1,x))$. These measurements then represent an approximation for the characteristic function of the flow domain.

\subsection{Velocity data}

The raw data acquired in phase contrast magnetic resonance velocimetry can be interpreted as 
\begin{equation*}
d^\delta|_{V_i} = \int_{V_i} \rho(x) e^{i 2\pi u(x)/v_{enc}} + \mbox{noise}, 
\end{equation*}
where $\rho$ is the propton density and $v_{enc}$ the velocity encoding parameter. An average value of the velocity in the voxel $V_i$ is then recovered as the phase of this signal, i.e., 
\begin{equation} \label{eq:retrieval}
u^\delta|_{V_i} = v_{enc} \cdot \arg(d^\delta|_{V_i}). 
\end{equation}
Phase unwrapping may be required, if the maximal velocity is larger then $v_{enc}$. In the absence of data noise and assuming that the density is given by $\rho = c \chi_\Omega$, one can use Taylor approximation to deduce that 
\begin{equation*}
u^\delta|_{V_i} \approx \frac{1}{|V_i \cap \Omega|} \int_{V_i \cap \Omega} u(x) dx,
\end{equation*}
where $\Omega$ is the flow domain and $u$ the true flow velocity. 
This is the measurement model used in the numerical tests. 
Let us note that the phase retrieval in \eref{eq:retrieval} is 
particularly sensitive to data perturbations, if $\textrm{abs}(d^\delta|_{V_i})$ is small, which is the case close to the boundary and outside the flow domain. 
Therefore, particularly large noise in the velocity data is expected for these voxels; compare with Figure~\ref{fig:vel_raw}.

\section*{Bibliography}


\end{document}